\documentclass{article}
 
\usepackage{amsthm}
\usepackage{amsfonts}
\usepackage{amssymb}
\usepackage{verbatim}
\usepackage[all]{xy}
\usepackage{amsmath, amscd, amsthm}
\usepackage{longtable}
\usepackage{amscd}
\usepackage{mathrsfs}
\usepackage{graphicx}
\usepackage{epstopdf}
\usepackage{latexsym}

\newtheorem{theorem}{Theorem}[section]
\newtheorem{lemma}[theorem]{Lemma}

\newtheorem{prop}[theorem]{Proposition}
\newtheorem{cor}[theorem]{Corollary}
\newtheorem{ex}[theorem]{Example}
\newtheorem{defn}[theorem]{Definition}
\newtheorem{rem}[theorem]{Remark}
\newtheorem{question}[theorem]{Question}

\usepackage{latexsym}

\newcommand{\z}{{\mathbb Z}}
\newcommand{\q}{{\mathbb Q}}

\newcommand{\zb}{\z[\langle b \rangle]}
\newcommand{\qb}{\q[\langle b \rangle]}

\newcommand{\gp}{\textrm{gp}}
\newcommand{\wre}{\textrm{ wr }}

\DeclareMathOperator{\cay}{Cay}
\DeclareMathOperator{\supp}{supp}
\DeclareMathOperator{\spann}{route}
\DeclareMathOperator{\trace}{reach}
\begin{document}

\title{Subgroup Distortion in Wreath Products of Cyclic Groups}
\author{Tara C. Davis, Alexander Yu. Olshanskii}
\maketitle

\begin{abstract}
We study the effects of subgroup distortion in the wreath products $A \wre \z$, where $A$ is finitely generated abelian. We show that every finitely generated subgroup of $A \wre \z$ has distortion function equivalent to some polynomial. Moreover, for $A$ infinite, and for any polynomial $l^k$, there is a $2$-generated subgroup of $A \wre \z$ having distortion function equivalent to the given polynomial. Also a formula for the length of elements in arbitrary wreath product $H \wre G$ easily shows that the group $\z_2 \wre \z^2$ has distorted subgroups, while the lamplighter group $\z_2\wre\z$ has no distorted (finitely generated) subgroups.
 
\end{abstract}

\let\thefootnote\relax
\let\thefootnote\relax


\section{Introduction}
The notion of subgroup distortion was first formulated by Gromov in \cite{gromov}. For a group $G$ with finite generating set $T$ and a subgroup $H$ of $G$ finitely generated by $S$, the distortion function of $H$ in $G$ is $$\Delta_{H}^{G}(l) = \max \{ |w|_S : w \in H, |w|_T \leq l \},$$ where $|w|_S$ represents the word length with respect to the given generating set $S$, and similarly for $|w|_T$. This function measures the difference in the word metrics on $G$ and on $H$. 

As usual, we only study distortion up to a natural equivalence relation. For non-decreasing functions $f$ and $g$ on $\mathbb{N}$, we say that $f \preceq g$ if there exists an integer $C>0$ such that $f(l) \leq Cg(Cl)$ for all $l \geq 0$. We say two functions are equivalent, written $f \approx g$, if $f \preceq g$ and $g \preceq f$. 
When considered up to this equivalence, the distortion function becomes independent of the choice of finite generating sets. If the subgroup $H$ is infinite then the growth of the 
distortion function is at least linear, and therefore one does not extend the
equivalence classes using the equivalence defined by the inequality $f(l) \leq Cg(Cl) +Cn.$ A subgroup $H$ of $G$ is said to be undistorted if $\Delta^{G}_{H}(l) \approx l$. If a subgroup $H$ is not undistorted, then it is said to be distorted, and its distortion refers to the equivalence class of $\Delta_H^G(l)$. 

\begin{rem}\label{poa}
Suppose there exists a subsequence of $\mathbb{N}$ given by $\{l_i\}_{i \in \mathbb{N}}$ where $l_i < l_{i+1}$ for $i \geq 1$. If there exists $c>0$ such that $\frac{l_{i+1}}{l_i} \leq c$, for all $i \geq 1$, and $f(l_i) \geq g(l_i)$, then $f \succeq g$. 
\end{rem}

Here we study the effects of distortion in various subgroups of the wreath products $\z^k \wre \z$, for $0<k \in \z$, and more generally, in $A \wre \z$ where $A$ is finitely generated abelian. 

Note that wreath products $A \wre B$ where $A$ is abelian play a very important role in group theory for many reasons. Given any semidirect product $G=C\lambda D$ with abelian normal subgroup $C,$ then any two homomorphisms from $A \to C$ and $B \to D$ (uniquely) extend to a homomorphism from $A \wre B$ to $G.$ Also, if $B$ is presented as a factor-group $F/N$ of a $k$-generated free group $F,$ then the maximal extension $F/[N,N]$ of $B$ with abelian kernel is canonically embedded in $\z^k \wre B$ (see \cite{magnus}.) Wreath products of abelian groups give an inexhaustible source of examples and counter-examples in group theory.

For instance, the group $\z\wre \z$ is the simplest example of a finitely generated (though not finitely presented) group containing a free abelian group of infinite rank. In \cite{sapir} the group $\z \wre \z$ is studied in connection with diagram groups and in particular with Thompson's group. In the same paper, it is shown that for $H_d=( \cdots (\z \wre \z) \wre \z) \cdots \wre \z)$, where the group $\z$ appears $d$ times, there is a subgroup $K \leq H_d \times H_d$ having distortion function $\Delta_K^{H_d \times H_d}(l) \succeq l^d$. In contrast to the study of these iterated wreath products, here we obtain polynomial distortion of arbitrary degree in the group $\z \wre \z$ itself. In \cite{cleary} the distortion of $\z \wre \z$ in Baumslag's metabelian group is shown to be at least exponential, and an undistorted embedding of $\z \wre \z$ in Thompson's group is constructed. 

In this note, rather than embedding the group $\z \wre \z$ into larger groups, or studying multiple wreath products, we will study distorted and undistorted subgroups in the wreath products $A \wre \z$ with $A$ finitely generated abelian.  The main results are as follows. 

\begin{theorem}\label{x}
Let $A$ be a finitely generated abelian group.
\begin{enumerate}
\item For any finitely generated infinite subgroup $H \leq A \wre \z$ there exists $m \in \mathbb{N}$ such that the distortion of $H$ in $A \wre \z$ is $$\Delta_H^{A \wre \z}(l) \approx l^{m}.$$
\item If $A$ is finite, then $m=1$; that is, all subgroups are undistorted.
\item If $A$ is infinite, then for every $m \in \mathbb{N}$, there is a $2$-generated subnormal subgroup $H$ of $A \wre \z$ having distortion function $$\Delta_H^{A \wre \z}(l) \approx l^{m}.$$
\end{enumerate}
\end{theorem}


The following will be explained in Subsection \ref{fs}.

\begin{cor}\label{ot}
For every $m \in \mathbb{N}$, there is a $2$-generated subgroup $H$ of the free $n$-generated metabelian group $S_{n,2}$ having distortion function $$\Delta_H^{S_{n,2}}(l) \succeq l^{m}.$$
\end{cor}

\begin{cor}\label{xx}
If we let the standard generating set for $\z \wre \z$ be $\{ a, b \}$, then the subgroup $H_m= \langle b, [ \cdots [a,b],b], \cdots, b] \rangle$, where the commutator is $(m-1)$-fold, is $m-1$ subnormal, isomorphic to the whole group $\z \wre \z$, with distortion $l^m$. In particular the normal subgroup $\langle b, [a,b] \rangle$ has quadratic distortion.
\end{cor}

Corollary \ref{xx} is proved at the end of this paper. Because the subgroup $\langle [a,b], b \rangle$ of $\z \wre \z$ is normal, it follows by induction that the distorted subgroup $H_m$ is subnormal.

\begin{rem}\label{mmm}
There are distorted embeddings from the group $\z \wre \z$ into itself as a normal subgroup. For example, the map defined on generators by $b \mapsto b, a \mapsto [a,b]$ extends to an embedding, and the image is a quadratically distorted subgroup by Corollary \ref{xx}. By Lemma \ref{t2}, $\z \wre \z$ is the smallest example of a metabelian group embeddable to itself as a normal subgroup with distortion.
\end{rem}

\begin{cor}\label{sc}
There is a distorted embedding of $\z \wre \z$ into Thompson's group $F$. 
\end{cor}

Under the embedding of Remark \ref{mmm}, $\z \wre \z$ embeds into itself as a distorted subgroup. It is proved in \cite{sapir} that $\z \wre \z$ embeds to $F$. Therefore, Corollary \ref{sc} is true. 

It is interesting to contrast Theorem \ref{x} part $(2)$ with the following, which will be discussed in Section \ref{ppq}. Throughout this paper, we use the convention that $\z_n$ represents the finite group $\z/n\z$.

\begin{prop}\label{pqp}
The group $G=\z_n \wre \z^k$ for $n\ge 1$, has a finitely generated subgroup $H$  with distortion at least $l^{k}$.
\end{prop}

Some of the techniques to be introduced in this paper include some computations with polynomials. We will use the theory of modules over principal ideal domains in Section \ref{mt} to reduce the problem of subgroup distortion in $\z^k \wre \z$ to the consideration of certain $2$-generated subgroups in $\z \wre \z$. Every such subgroup is associated with a polynomial, and therefore we need to define and compute the distortion of arbitrary polynomial, as in Theorem \ref{tgs}. All of these techniques are used in conjunction with Theorem \ref{t6}, which provides a formula for computing the word length in arbitrary wreath product and makes computing subgroup distortion more tangible in the examples we consider. 

\section{Background and Preliminaries}
\subsection{Subgroup Distortion}

Here we provide some examples of distortion as well as some basic facts to be used later on.

\begin{ex}
\item 1. Consider the three-dimensional Heisenberg group $\mathcal{H}^3 = \langle a, b, c | c=[a,b], [a,c]=[b,c] = 1 \rangle.$ It has cyclic subgroup $\langle c \rangle_{\infty}$ with quadratic distortion, which follows from the equation $c^{l^2} = [a^l,b^l]$. 
\item 2. The Baumslag-Solitar Group $BS(1,2) = \langle a, b | b a b^{-1} = a^2 \rangle$ has cyclic subgroup $\langle a \rangle_{\infty}$ with at least exponential distortion, because $a^{2^l} =b^lab^{-l}.$
\end{ex}

However, there are no similar mechanisms distorting subgroups in $\z \wre \z$. Therefore, a natural conjecture would be that free metabelian groups or the group $\z \wre \z$ do not contain distorted subgroups. This conjecture was brought to the attention of the authors by Denis Osin. The result of Theorem \ref{x} shows that the conjecture is not true.

The following facts are well-known and easily verified. When we discuss distortion functions, it is assumed that the groups under consideration are finitely generated. 

\begin{lemma}\label{wkf}
\item 1. If $H \leq G$ and $[G:H] < \infty$ then $\Delta_H^G(l) \approx l$.
\item 2. If $H \leq K \leq G$ then $\Delta_{H}^K(l) \preceq \Delta_H^G(l).$

\item 3. If $H \leq K \leq G$ then $\Delta_{H}^G(l) \preceq \Delta_K^G((\Delta_H^K(l)).$
\item 4. If $H$ is a retract of $G$ then $\Delta_H^G(l) \approx l$.

\item 5. If $G$ is a finitely generated abelian group, and $H \leq G$, then $\Delta_H^G(l) \approx l$.
\end{lemma}

\subsection{Wreath Products}

We consider the wreath products $A \wre B$ of finitely generated groups $A=\gp \langle S \rangle = \langle \{y_1, \dots, y_s\} \rangle$ and $B = \gp \langle T \rangle = \langle \{ x_1, \dots, x_t \} \rangle$. We introduce the notation that $A \wre B$ is the  semidirect product $W\lambda B$, where $W$ is the direct product $\displaystyle\times_{g \in B}A_g,$ of isomorphic copies $A_g$ of the group $A.$  We view elements of $W$ as functions from $B$ to $A$ with finite support, where for any $f \in W$, the support of $f$ is $\supp(f)=\{g \in B: f(g) \ne 1\}$. The (left) action $\circ$ of $B$ on $W$ by automorphisms is given by the following formula: for any $f \in W, g \in B$ and $x \in B$ we have that $(g\circ f)(x)=f(xg)$.

Any element of the group $A \wre B$ may be written uniquely as $wg$ where $g \in B, w \in W$. The formula for multiplication in the group $A \wre B$ is given as follows. For $g_1, g_2 \in B, w_1, w_2 \in W$ we have that $(w_1g_1)(w_2g_2)=(w_1(g_1\circ w_2))(g_1g_2).$  In particular,
$B$ acts by conjugation on $W$ in the wreath product: $gwg^{-1}=g\circ w.$ 

Therefore the wreath product is generated by the subgroups  $B$ and $A_1\le W,$ where non-trivial
functions from $A_1$ have support $\{1\}.$ In what follows,
the subgroup $A_1$ is identified with $A,$ and so $A_g=gAg^{-1},$ and $S\cup T$ is a
finite set of generators in $A \wre B.$  In particular, $\z \wre \z$ is generated by
$a$ and $b$ where $a$ generates the left (passive) infinite cyclic group and $b$ generates
the right (active) one.

Here we observe that a finitely generated abelian subgroup of $G=A \wre B $ with finitely generated abelian $A$ and $B$ is undistorted. It should be remarked that the authors are aware that  the proof of the fact that abelian subgroups of $\z^k \wre \z$  are undistorted is available in \cite{sapir}. In that paper it is shown that $\z^k \wre \z$ is a subgroup of the Thompson group $F$, and that every finitely generated abelian subgroup of $F$ is undistorted. However, our observation is elementary and so we include it.

\begin{lemma}\label{abel} Let $A$ and $B$ be finitely generated abelian groups. Then
every finitely generated abelian subgroup $H$ of  $ A \wre B$ is 
undistorted. 
\end{lemma}

\begin{proof}
It follows from the classification of  finitely generated  abelian groups $G$ that every subgroup $S$ is a retract of a subgroup of finite index in $G$, and so we are done if $H$ is a subgroup of $A$ or $B$, or if $H \cap W = \{1\}$, by Lemma \ref{wkf}. Therefore we assume that
$H \cap W \ne \{1\}.$ Since $H$ is abelian, this implies that the the factor-group
$HW/W$ is finite. Then it suffices to prove the lemma for $H_1=H\cap W$ since $[H:H_1]\le \infty.$
Because $H_1$ is finitely generated, it is contained in a finite product of conjugate copies of $A$. That is to say, $H_1 \subset A'$ for a wreath product $A' \wre B' = W\lambda B'$ where $B'$ has finite index in $B$. We are now reduced to our earlier argument, thus completing the proof.
\end{proof}

\begin{rem} In fact, under the assumptions of Lemma \ref{abel}, $H$ is a retract
of a subgroup having finite index in $A \wre B.$
\end{rem}

We now return to one of the motivating ideas of this paper, and complete the explanation of Remark \ref{mmm}.

\begin{lemma}\label{t2}
The group $\z \wre \z$ is the smallest metabelian group which embeds to itself as a normal distorted subgroup in the following sense. For any metabelian group $G$, if there is an embedding $\phi: G \rightarrow G$ such that $\phi(G) \unlhd G$ and $\phi(G)$ is a distorted subgroup in $G$, then there exists some subgroup $H$ of $G$ for which $H \cong \z \wre \z.$
\end{lemma}

\begin{proof}
By Lemma \ref{wkf}, we have that the group $G/\phi(G)$ is infinite, else $\phi(G)$ would be undistorted. Being a finitely generated solvable group, $G/\phi(G)$ must have a subnormal factor isomorphic to $\z$. Because $\phi(G) \cong G$, one may repeat this argument to obtain a subnormal series in $G$ with arbitrarily many infinite cyclic factors. Therefore, the derived subgroup $G'$ has infinite (rational) rank.

Since the group $B=G/G'$ is finitely presented, the action of $B$ by conjugation makes $G'$ a finitely generated left $B$ module. Hence, $G'=\langle B \circ C\rangle$ for some finitely generated $C \leq G'$. Because it is a finitely generated abelian group, $B=\langle b_k \rangle \cdots \langle b_1 \rangle$ is a product of cyclic groups. Therefore for some $i$ we have a subgroup $A=\langle \langle b_{i-1} \rangle \cdots \langle b_1 \rangle \circ C \rangle$ of finite rank in $G'$ but $\langle \langle b_i \rangle \circ A \rangle$ has infinite rank. Then $A$ has an element $a$ such that the $\langle b_i \rangle$-submodule generated by $a$ has infinite rank, and so it is a free $\langle b_i \rangle$-module. It follows that $a$ and $b$, where $b_i=bG'$, generate a subgroup of the form $\z \wre \z$.
\end{proof}

\subsection{Connections with Free Solvable Groups}\label{fs}
In \cite{magnus}, Magnus shows that if $F=F_k$ is an absolutely free group of rank $k$ with normal subgroup $N$, then the group $F/[N,N]$ embeds into $\z^k \wre F/N= \z^k \wre G$.
This wreath product is a semidirect product $W\lambda G$ where the action of $G$ by
conjugation turns $W$ into a free left $\z[G]$-module with $k$ generators.
For more information in an easy to read exposition, refer to \cite{rem}.

\begin{rem}
The monomorphism $\alpha : F/[N,N] \rightarrow \z^k \wre G$ is called the Magnus embedding.
\end{rem}

We let $S_{k,l}$ denote the $k$-generated derived length $l$ free solvable group.

\begin{lemma}\label{y}
If $k,l \geq 2$, then the group $S_{k,l}$ contains a subgroup isomorphic to $\z \wre \z$. 
\end{lemma}

\proof It is well known (and follows from the Magnus embedding) that any nontrivial $a \in S_{k,l}^{(l-1)}$ and $b \notin S_{k,l}^{(l-1)}$ generate $\z \wre \z$.\endproof

\medskip

It should be noted that by results of \cite{shmelkin}, the group $\z \wre \z^2$ can not be embedded into any free metabelian or free solvable groups.

\medskip

Subgroup distortion has connections with the membership problem.
It was observed in \cite{gromov} and proved in \cite{farb} that for a finitely generated subgroup $H$ of a finitely generated group $G$ with solvable word problem, the membership problem is solvable in $H$ if and only if the distortion function $\Delta_{H}^{G}(l)$ is bounded by a recursive function. 

By Theorem 2 of \cite{umirbaev}, the membership problem for free solvable groups of length greater than two is undecidable. Therefore, because of the connections between subgroup distortion and the membership problem just mentioned, we restrict our primary attention to the case of free metabelian groups. It is worthwhile to note that the membership problem for free metabelian groups is solvable (see \cite{romanovskii}).

Lemma \ref{y} motivates us to study distortion in $\z \wre \z$ in order to better understand distortion in free metabelian groups. Distortion in free metabelian groups is similar to distortion in wreath products of free abelian groups, by Lemma \ref{y} and the Magnus embedding. In particular, if $k \geq 2$ then $$\z \wre \z \leq S_{k,2} \leq \z^k \wre \z^k.$$ Thus by Lemma \ref{wkf}, given $H \leq \z \wre \z$ we have $$\Delta_{H}^{\z \wre \z}(l) \preceq \Delta_{H}^{S_{k,2}}(l).$$ This explains Corollary \ref{ot}. On the other hand, given $L \leq S_{k,2}$ then we have $$\Delta_{L}^{S_{k,2}}(l) \preceq \Delta_{L}^{\z^k \wre \z^k}(l).$$ Based on this discussion, we ask the following. An answer would be helpful in order to more fully understand subgroup distortion in free metabelian groups.

\begin{question}
\item What effects of subgroup distortion are possible in $\z^k \wre \z^k$ for $k>1$? 
\end{question}

\section{Canonical Forms and Word Metric}\label{s2}

Here we aim to further understand how the length of an element of a wreath product
$A\wre B$ depends on the canonical form of this element. 
Let us start with $G=\z^k \wre \z=W\lambda\langle b \rangle,$ where $\z^k=\gp\{a_1,\dots,a_k\}.$

Because the subgroup $W$ of $\z^k \wre \z = W \lambda \z$ is abelian, we also use additive notation to represent elements of $W$.

\begin{rem}\label{yyy}
In the case of $\z \wre \z=\langle a \rangle \wre \langle b \rangle$, we use module language to write any element as $$w={\sum_{i=-\infty}^{\infty}}  m_i (b^i \circ a)=f(b)a=f(x)a \textrm{, where } f(x)={\sum_{i=-\infty}^{\infty}} m_ix^i$$ is a Laurent polynomial in $x$, and the sums are finite.
Similarly, if $A$ is a finitely generated abelian group, then by the definition of $A \wre \z,$ 
arbitrary element in $A \wre \z$ is 
of the form $$wb^t=\left(\sum_{i=1}^kf_i(x)a_i \right) b^t,$$ where $f_i(x)$ are Laurent polynomials and $a_i$-s generate $A$. This form is unique if $A=\z^k$ is a free abelian
with basis $a_1,\dots,a_k$.
\end{rem}

We will use the notation that $(w)_i$ equals the conjugate $bwb^{-i}$ for 
$i \in \z$ and $w \in W.$
The normal form described in Remark \ref{yyy} for elements of $A\wre \z$
is necessary to obtain a general formula for computing the word length. 

\begin{rem}\label{yy}
Arbitrary element of $A \wre \z$ may be written in a normal form, following \cite{ct}, as
$$\left( (u_1)_{\iota_1}+ \cdots +(u_N)_{\iota_N}+(v_1)_{-\epsilon_1}+ \cdots +(v_M)_{-\epsilon_M}
\right)b^t$$ 
where  $0 \leq \iota_1 < \cdots <\iota_N, 0 < \epsilon_1 < \cdots < \epsilon_M$, and $u_1, \dots, u_N, v_1, \dots, v_M$ are elements in $A\backslash\{1\}.$
\end{rem}

The following formula for the word length in $A \wre \z$ is given in \cite{ct}.

\begin{lemma}\label{t5}
Given an element in $A \wre \z$ having normal form as in Remark \ref{yy}, its length is given by the formula $$\sum_{i=1}^N |u_i|_{A} + \sum_{i=1}^M |v_i|_{A} + \min \{ 2 \epsilon_M +\iota_N +|t-\iota_N|, 2\iota_N+\epsilon_M+|t+\epsilon_M|\}.$$ 
where $|*|_A$ is the length in the group $A.$ 
\end{lemma}

The formula from Lemma \ref{t5} becomes more intelligible if one extends it to
wreath products $A \wre B$ of arbitrary finitely generated groups. We want to
obtain such a generalization in this section since we consider non-cyclic active
groups in Section \ref{ppq}. We fix the notation that, with respect to the symmetric generating set $T=T^{-1}$, the Cayley graph $\cay(B)$ is defined as follows. The set of vertices is all elements of $G$. For any $g \in G, t \in T$, $g$ and $gt$ are joined by an edge pointing from $g$ to $gt$ whose label is $t$.

Any $u \in A \wre B$ can be expressed as follows:
\begin{equation}\label{r8}
 (b_1\circ a_1)\dots (b_r\circ a_r) g
\end{equation} 
where $g \in B, w=(b_1\circ a_1)\dots (b_r\circ a_r)\in W, 1 \ne a_j \in A, b_j \in B$ and for $i \ne j$ we have $b_i \ne b_j$. The expression $(\ref{r8})$ is unique, up to a rearrangement of the (commuting) factors $b_j\circ a_j$.

For any $u=wg \in A \wre B$ with canonical form as in Equation $(\ref{r8})$ we consider the set $P$ of paths in the Cayley graph $\cay(B)$ which start at $1$, go through every vertex $b_1, \dots, b_r$ and end at $g$. We introduce the notation that $$\trace(u)=\min\{||p||: p \in P\},$$ $$\spann(u)=\textrm{ the particular } p \in P \textrm{ realizing } \trace(u)=||p||.$$ We also define the norm of any such representative $w$ of $W$ by $$||w||_A= \sum_{j=1}^r|a_j|_S.$$ 

We have the following formula for word length, which generalizes that given for the case where $B= \z$ in the paper \cite{ct}. ({\bf Caution}: The {\it right}-action definition of wreath product
would be incompatible with the standard definition of Cayley graph in the proof of Theorem \ref{t6}.)

\begin{theorem}\label{t6}
For any element $u=wg \in A \wre B$, we have that $$|wg|_{S,T} = ||w||_A+\trace(u)$$ where $u=(b_1\circ a_1)\dots (b_r\circ a_r) g$ is the canonical form of Equation $(\ref{r8})$.
\end{theorem}

\proof We will use the following pseudo-canonical (non-unique) form in the proof. This is just the expression of Equation $(\ref{r8})$ but without the assumption that all $b_j$ are distinct
or $a_j$-s are non-trivial. 
 
For any element $u \in A \wre B$ which is expressed in pseudo-canonical form we may define a quantity depending on the given factorization by $$\Psi((b_1\circ a_1)\dots (b_r\circ a_r) g)=\sum_{j=1}^r|a_j|_S+|b_1|_T +|b_1^{-1}b_2|_T+\dots + |b_{r-1}^{-1}b_r|_T+|b_r^{-1}g|_T.$$ 

First we show that for $u$ in canonical form $(\ref{r8}),$ it holds that $|u|_{S,T} \geq ||w||_A+\trace(u)$.

By the choice of generating set $\{S,T\}$ of $A \wre B$, we have that any element $u \in A \wre B$ may be written as
\begin{equation}\label{t7}
u=g_0h_1g_1 \cdots h_mg_m
\end{equation}
 where $m\ge 0, g_i\in B, h_j\in A, g_0$ and $g_m$ can be trivial, but all other factors are non-trivial. We may choose the expression \eqref{t7} so that $|u|_{S,T}=\sum_{j=1}^m|h_j|_S+\sum_{i=0}^m|g_i|_T .$ Observe that we may use the expression from Equation $(\ref{t7})$ to write 
 \begin{equation}\label{r7}
 u=(x_1\circ h_1) \dots (x_{m}\circ h_m)g
 \end{equation}
 where $g=g_0 \dots g_m$  and $x_j=g_0\dots g_{j-1}$, for $j=1, \dots, m$.
 
Then we have by definition that for the pseudo-canonical form $(\ref{r7})$,
\begin{equation}\label{ppd}
\Psi((x_1\circ h_1) \dots (x_{m}\circ h_m)g)=\sum_{j=1}^m|h_j|_S+|x_1|_T+ |x_1^{-1}x_2|_T +\dots + |x_{m-1}^{-1}x_m|_T +|x_m^{-1}g|_T$$
$$=\sum_{j=1}^m|h_j|_S+\sum_{i=0}^m|g_i|_T = |u|_{S,T}.
\end{equation}

It is possible that in the form of Equation $(\ref{r7})$, some $x_i = x_j$ for $1 \leq i \ne j \leq m$. 
When taking $u$ to the canonical form $wg=(b_1\circ a_1) \dots (b_{r}\circ a_r)g$ of Equation $(\ref{r8})$, we claim that 
\begin{equation}\label{one}
||w||_A \leq \displaystyle\sum_{j=1}^m|h_j|_S
\end{equation} and that 
\begin{equation}\label{two}
\trace(u) \leq |x_1|_T+ |x_1^{-1}x_2|_T+\dots + |x_{m-1}^{-1}x_m|_T +|x_m^{-1}g|_T.
\end{equation}
Obtaining the canonical form requires a finite number of steps of the following nature. We take an expression such as $$(x_1\circ h_1) \dots (x_i\circ h_i)\dots (x_i\circ h_j)\dots (x_{m}\circ h_m)$$ and replace it with 
$$(x_1\circ h_1) \dots (x_i\circ h_ih_j)\dots (x_{j-1}\circ h_{j-1})(x_{j+1}\circ h_{j+1})\dots (x_{m}\circ h_m).$$ The assertion of Equation $(\ref{one})$ follows because $$|h_ih_j|_S \leq |h_i|_S+|h_j|_S.$$ Equation $(\ref{two})$ is true because $$|x_{j-1}^{-1}x_{j+1}|_T \leq |x_{j-1}^{-1}x_j|_T+|x_j^{-1}x_{j+1}|_T,$$ which implies that 
$$|b_1|_T +|b_1^{-1}b_2|_T+\dots + |b_{r-1}^{-1}b_r|_T+|b_r^{-1}g|_T$$ $$\le |x_1|_T +|x_1^{-1}x_2|_T+\dots + |x_{m-1}^{-1}x_m|_T+|x_m^{-1}g|_T.$$
Finally, we have that $$\trace(u) \leq |b_1|_T +|b_1^{-1}b_2|_T+\dots + |b_{r-1}^{-1}b_r|_T+|b_r^{-1}g|_T,$$ because the right hand side is the length of a particular path in $P$: the path which travels from $1$ to $b_1$ to $b_2, \dots,$ to $b_r$ to $g$. It follows that the length of this path is at least as large as the length of $\spann(u)$.

Thus for a canonical form $u=(b_1\circ a_1)\dots (b_r\circ a_r) g$ we see by Equations $(\ref{ppd})$, $(\ref{one})$ and $(\ref{two})$ that $$||w||_A + \trace(u) \leq \Psi((x_1\circ h_1) \dots (x_{m}\circ h_m)g)=|u|_{S,T}.$$

 To obtain the reverse inequality, take $u=(b_1\circ a_1) \dots (b_{r}\circ a_r)g$ in $A \wre B$ in canonical form. By the definition, $\spann(u)$ will be a path that starts at $1$, goes in some order directly through all of $b_1, \dots, b_r$, and ends at $g.$ 

We may rephrase this to say that for some $\sigma \in \textrm{Sym}(r),$ there is a path $p=\spann(u)$ in $P$ such that $|p|_T=|b_{\sigma(1)}|_T+ |b_{\sigma(1)}^{-1}b_{\sigma(2)}|_T+... + |b_{\sigma(r-1)}^{-1}b_{\sigma(r)}|_T +|b_{\sigma(r)}^{-1}g|_T$. In other words, 
$$\trace(u)=|b_{\sigma(1)}|_T+ |b_{\sigma(1)}^{-1}b_{\sigma(2)}|_T+... + |b_{\sigma(r-1)}^{-1}b_{\sigma(r)}|_T +|b_{\sigma(r)}^{-1}g|_T.$$
Moreover, in the wreath product we have that $$u=(b_{\sigma(1)}\circ a_{\sigma(1)}) \cdots (b_{\sigma(r)}\circ a_{\sigma(r)})g$$ $$=b_{\sigma(1)}a_{\sigma(1)}b_{\sigma(1)}^{-1}b_{\sigma(2)}a_{\sigma(2)} \cdots b_{\sigma(r-1)}^{-1}b_{\sigma(r)}a_{\sigma(r)}b_{\sigma(r)}^{-1}g.$$ This implies that $$|u|_{S,T} \leq |b_{\sigma(1)}|_T+|a_{\sigma(1)}|_S+|b_{\sigma(1)}^{-1}b_{\sigma(2)}|_T+ \cdots +|a_{\sigma(r)}|_S+|b_{\sigma(r)}^{-1}g|_T$$ $$= \sum_{j=1}^r|a_j|_S+\trace(u)=||w||_A+\trace(u).$$
\endproof

\section{Distortion in $\z_p \wre \z^k$}\label{ppq}

We begin with the following result, the proof of which exploits the formula of Theorem \ref{t6}.

\begin{prop}\label{t3}
The group $\z_2 \wre \z^2$ contains distorted subgroups.
\end{prop}

This is interesting in contrast to the case of $\z_2 \wre \z$ which has no effects of subgroup distortion. The essence in the difference comes from the fact that the Cayley graph of $\z$ is one-dimensional, and that of $\z^2$ is asymptotically two-dimensional, which gives us more room to create distortion using Theorem \ref{t6}.

We will use the following notation in the case of $G=\z_2 \wre \z^2$: $a$ generates the passive group of order $2$ while $b$ and $c$ generate the active group $\z^2.$ 

The canonical form of Equation $(\ref{r8})$ will be denoted by $$((g_1+ \cdots +g_k) a)g$$ for $g_1, \dots, g_k$ distinct elements of $\z^2$ and $g \in \z^2$. We may do this because any nontrivial element of $\z_2$ is just equal to the generator $a$. 

\begin{lemma}\label{ss}
 Let $H \leq G$ be generated by a nontrivial element $w \in W$ as well as the generators $b,c$ of $\z^2$. Then $H \cong G$.
\end{lemma}

We know that $W=\displaystyle\bigoplus_{g \in \z^2}\langle g\circ a\rangle$ is a free module over $\z_2[\z^2]$. Therefore, we may think of $W$ as being the Laurent polynomial ring in two variables, say, $x$ for $b$ and $y$ for $c$. We can use the module language to express any element as $w=f(x,y)a=(x^{i_1}y^{j_1}+ \cdots+x^{i_k}y^{j_k})a$, where for $p \ne q$ we have that $x^{i_p}y^{j_p} \ne x^{i_q}y^{j_q}$. This corresponds to the canonical form $w=(g_1+ \cdots +g_k)a$ where $g_p=b^{i_p}c^{j_p}$ for $p=1, \dots, k$. 

We now have all the required facts to prove Proposition \ref{t3}.

\begin{proof} of Proposition \ref{t3}:
Let $G=\z_2 \wre \z^2 = \gp \langle a, b, c \rangle$ and $H = \gp \langle b, c, w \rangle$ where $w=[a,b]=(1+x)a$. By Lemma \ref{ss} we have that $H \cong G$. Let $$f_l(x)=\displaystyle\sum_{i=0}^{l-1}x^i \textrm{ and } g_l(x)=(1+x)f_l(x).$$ 
The element $f_l(x)f_l(y)w \in H$ is in canonical form, when written in the additive group notation as $\sum_{i,j=0}^{l-1} b^ic^j\circ w.$

By Theorem \ref{t6}, we have that its length in $H$ is at least $l^2+l^2$ since
the support of it has cardinality $l^2$, and the length of arbitrary loop going through $l^2$ different vertices is at least $l^2$. 

\begin{center}
\small{Figure $1$: The $l^2$ vertices (left) and the rectangle with perimeter $2l+2(l-1)$ (right)}
\includegraphics[scale=.58]{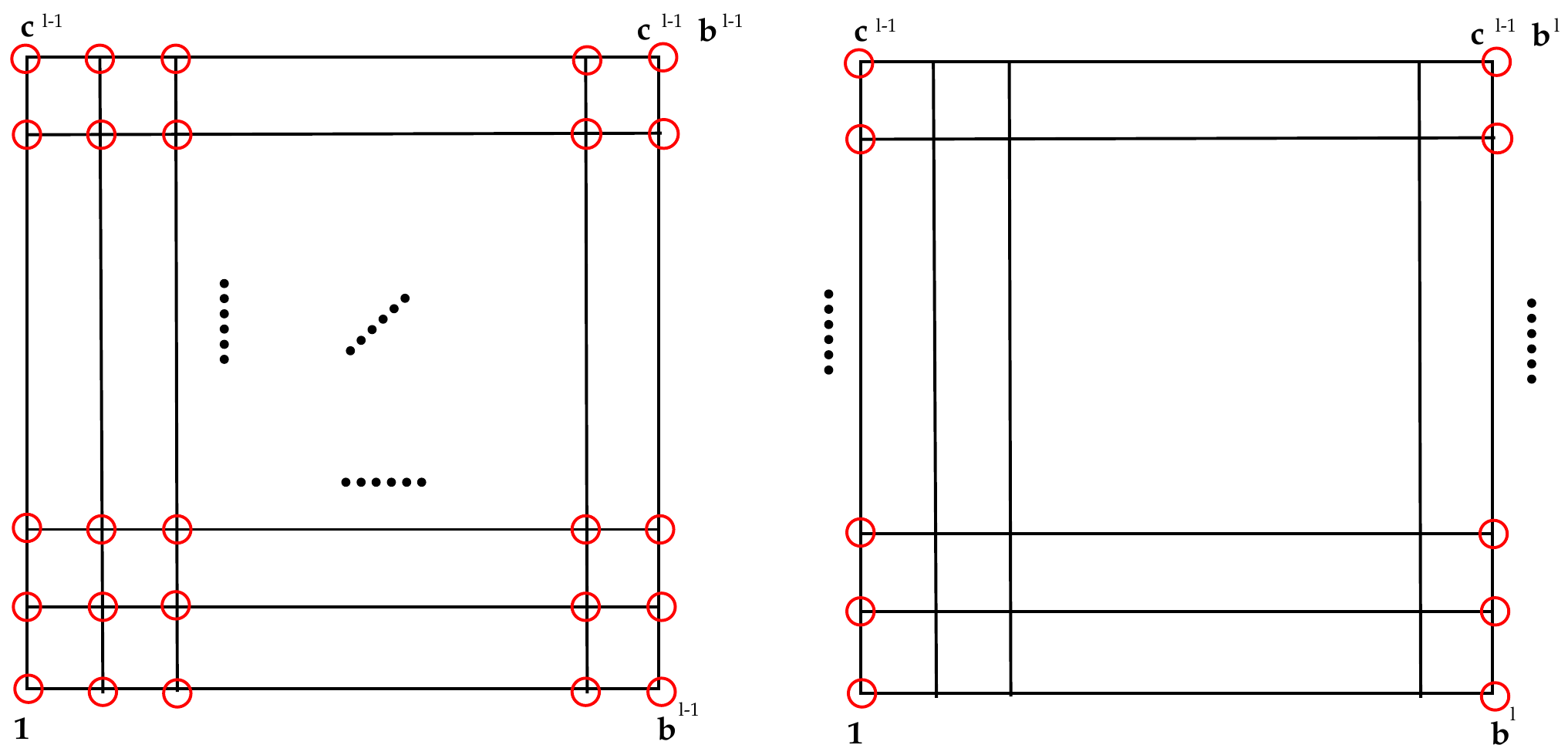}\\
\end{center}

Now we compute the length of $f_l(x)f_l(y)w$ in $G$. We have that $$f_l(x)f_l(y)w=(1+x)f_l(x)f_l(y)a=g_l(x)f_l(y)a=\bigg[\sum_{i=0}^{l-1} (y^i+y^ix^{l})\bigg]a.$$ Theorem \ref{t6} shows that $|f_l(x)f_l(y)w|_G=2l+2(l-1)+2l$. This is because the shortest path in $\cay(\z^2)$ starting at $1$, passing through $1, c, \dots, c^{l-1}$ and $b^l, cb^l, \dots, c^{l-1}b^l$ and ending at $1$ is given by traversing the perimeter of the rectangle, and so gives the length of $2(l-1)+2l$. 

Therefore the subgroup $H$ is at least quadratically distorted. 
\end{proof}

\begin{rem}
The subgroup $H$ is not normal in $G$  because the element $aca^{-1}$ is not in $H$. 
\end{rem}

The proof of Proposition \ref{t3} can be generalized as follows. Consider the group $G=\z_n \wre \z^k=\gp \langle a, b_1, \dots, b_k\rangle$ for $n\ge 2$ and $k>1$. Then the subgroup $H = \gp \langle w, b_1, \dots, b_k\rangle$ where $w=(1-x_1) \cdots (1-x_{k-1})a=[...[a,b_1],b_2],...b_{k-1}]$ has distortion at least $l^{k}$. This is a restatement of Proposition \ref{pqp}.

By (the analogue of) Lemma \ref{ss} we have that $H \cong G$ and so we can compute lengths using Theorem \ref{t6}. Consider the element $f_l(x_1) \cdots f_l(x_k)w$ in $H$. Then it has length in $H$ at least equal to $l^k+l^k$ because the path in $\cay(\z^k)$ arising from Theorem \ref{t6} would need to pass through at least $l^k$ vertices: ${b_1}^{\alpha_1} \cdots {b_k}^{\alpha_{k}}$ for $\alpha_i \in \{0, \dots, l-1\}, i=1, \dots, k$. In the group $G$, $$f_l(x_1) \cdots f_l(x_k)w=g_l(x_1) \cdots g_l(x_{k-1})f_l(x_k)a.$$ This has linear length, which follows because the vertices of the support are placed along the edges of a 
$k$-dimensional parallelotope, such that the length of any edge of the parallelotope is at most $l$.

\section{Estimating Word Length}
Although the notion of equivalence has only been defined for functions from $\mathbb{N}$ to $\mathbb{N}$, we would like to define a notion of equivalence for functions on a finitely generated group. We say that two functions $f, g: G \rightarrow \mathbb{N}$ are equivalent if there exists $C>0$ such that for any $x \in G$ we have $$\frac{1}{C}f(x)-C \leq g(x) \leq Cf(x)+C.$$ If there is a function $f: G \rightarrow \mathbb{N}$ such that $f \approx | \cdot |_G$, then for any subgroup $H$ of $G$, $\Delta_H^G(l) \approx \max \{ |x|_H : x \in H, f(x) \leq l\}$.

We need to establish a looser way of estimating lengths in $\z \wre \z$,  than the formula introduced in Lemma \ref{t5}. 
Recall that this group has standard generators $a\in W$ (passive)
and $b$ (active). 

Here we call {\it exemplary} any sugroup $H=\langle b,w \rangle \leq \z \wre \z$ where
$w \in W \backslash 1.$ We have $w=h(x)a,$ where $h(x)= \sum_{j=0}^{t}d_{j}x^j,$ and $d_0\ne 0.$
This follows without loss of generality by conjugating of $w$ by a power of $b$. Thus
we associate a polynomial $h(x)\in\z[x]$ with any exemplary subgroup $H.$

 \begin{lemma}\label{xxx} The mapping
$a\mapsto w, b\mapsto b$ extends to a monomorphism of the wreath product $\z \wre \z$  onto the exemplary subgroup
 \end{lemma}
 
\proof This follows because in this case $W$ is a free  module with one generator $a$ over the
domain $\z[\langle b\rangle],$ $w=h(x)a,$ and the mapping
$u\to h(x)u$ ($u\in W$) is an injective module homomorphism. \endproof





Then for any element $g \in H$, we may write $g= f(x)w=(f(x)h(x)a)b^n,$ where $f(x)=\sum_{q=s}^{s+p}z_qx^q$ is
a Laurent polynomial. Denote by $S(f)$ the sum $\sum_{q=s}^{s+p}|z_q|.$
For this element, consider the norms 
$$e(g) = S(fh) \textrm{ and } e_H(g) = S(f)$$


Letting $\iota(g)=\max\{t+s+p,0\}, \varepsilon(g)=\min\{s,0\}, \iota_H(g)=\max\{s+p,0\},$ 
$\varepsilon_H(g)=\min\{s,0\},$ we define $u_H(g)=\iota_H(g)-\varepsilon_H(g) \textrm{ and } u(g)=\iota(g)-\varepsilon(g).$

Consider the function $$\delta(l)= \max \{e_H(g): g \in H \cap W, e(g) \leq l \textrm{ and } u(g) \leq l\}.$$ The following Lemma shows that we may simplify computations of word length in exemplary subgroups.

\begin{lemma}\label{r2}
Let $H=\langle b, w \rangle \leq \z \wre \z$ be an exemplary subgroup.  Then we have that $$\Delta_H^{\z \wre \z}(l) \approx \delta(l).$$
\end{lemma}

\begin{proof}
Recall that by Lemma \ref{t5}, we have the following formulas. For $g \in H$ with the notation established above, we have that:
$$|g|_H=e_H(g)+\min\{-2\varepsilon_H(g)+\iota_H(g)+|n-\iota_H(g)|, 2\iota_H(g)-\varepsilon_H(g)+|n-\varepsilon_H(g)|\}$$ and 
$$|g|_{\z \wre \z}=e(g)+\min\{-2\varepsilon(g)+\iota(g)+|n-\iota(g)|, 2\iota(g)-\varepsilon(g)+|n-\varepsilon(g)|\}.$$

The following inequality follows from the definitions: 
\begin{equation}\label{bb}
\max \{e(g), u(g), |n| \} \leq |g|_{\z^r \wre \z}.
\end{equation}

Similarly, we have that
\begin{equation}\label{bbb}
|g|_H \leq e_H(g)+2u_H(g)+|n| \textrm{ and } |g|_{\z \wre \z} \leq e(g)+2u(g)+|n|.
\end{equation}

Observe that for $g \in H \cap W$ we have that 
\begin{equation}\label{c}
|g|_H \geq \max\{e_H(g), u_H(g)\}.
\end{equation}

Observe that 
\begin{equation}\label{cc}
\max \{u_H(g) : g \in H, u(g) \leq l\} \leq l.
\end{equation}   
   
Thus, $$\Delta_H^{\z\wre \z}(l) \leq \max \{e_H(g):g \in H, e(g) \leq l, u(g) \leq l\}$$ $$+ \max \{2u_H(g) : g \in H, u(g) \leq l\}+ \max \{|n| :g \in H, |n| \leq l\} \leq \delta(l)+3l.$$ 
The first inequality follows from Equation (\ref{bb}), the second from Equation (\ref{bbb}).

On the other hand, we have that $$\Delta_H^{\z \wre \z}(l) \geq \max\{e_H(g): g \in H \cap W, e(g) \leq l/4, u(g) \leq l/4\}$$ $$-\max\{u_H(g): g \in H \cap W, e(g) \leq l/4, u(g) \leq l/4\} \geq \delta(l/4)-l/4.$$

The first inequality follows from Equation (\ref{bbb}), the second from Equation (\ref{c}), and the third from Equation (\ref{cc}).

Thus $\Delta_H^{\z \wre \z}(l)$ and $\delta(l)$ are equivalent.
\end{proof}

\section{Distortion of Polynomials}

In order to understand distortion in exemplary subgroups of $\z \wre \z$, we will introduce the notion of  distortion of a polynomial. 

\begin{defn}\label{pd}
Let $R$ be a subring of a field with a real valuation, and consider the polynomial ring $R[x]$. We will define the norm function $S: R[x] \rightarrow \mathbb{R}^+$ which takes any $f(x)=\sum_{i=0}^na_ix^i \in R[x]$ to $S(f)=\sum_{i=0}^n|a_i|.$ For any $h \in R[x]$ and $c>0$, we define the distortion of the polynomial $h$ from $\mathbb{N}$ to $\mathbb{N}$ by 
\begin{equation}\label{3811}
\Delta_{h,c}(l) = \sup\{ S(f) : \deg(f) \leq cl, \textrm{ and } S(hf) \leq cl \}.
\end{equation}
\end{defn}

\begin{rem}
Taking into account the inequality $S(hf) \le cl$, one can easily find some explicit upper boundes $C_i=C_i(h,c,l)$ for the modules of the coefficient at $x^i$ of $f(x)$ in Formula \eqref{pd}, starting with the lowest coefficients. Therefore the supremum in Equation \eqref{3811} is finite. Furthermore, if $R = \z,  \mathbb{R}$ or $\mathbb{C}$ then the supremum is taken over a compact set of polynomials of bounded degree with bounded coefficients, and since $S$ is a continuous function, one may replace $\sup$ by $\max$ in Definition \ref{pd}.
\end{rem}

Note that the distortion does not depend on the constant $c$, up to equivalence, and so we will consider $\Delta_h(l)$. 


The following fact makes concrete our motivation for studying distortion of polynomials.

\begin{lemma}\label{hpu}
Let $H$ be an exemplary subgroup  
 $\langle b, w \rangle \leq \z \wre \z,$ and $w=h(x)a$ for $h=d_0+\cdots+d_tx^t  \in \z[x]$. Then $$\Delta_h(l) \approx \Delta_H^{\z \wre \z}(l).$$
\end{lemma}



\begin{proof}
By Lemma \ref{r2}, we have that $\Delta_H^{\z \wre \z}(l) \approx \delta(l)=\max\{e_H(g):g \in H \cap W, e(g) \leq l, u(g) \leq l\}.$ Let $g_l=f_l(x)w \in H \cap W$ be so that $\delta(l)=e_H(g_l)$. 
There exists $n \in \z$ so that $\bar g_l=b^ng_lb^{-n} \in H$ and $\bar g_l=\bar f_l(x)w$ where $\bar f_l(x)$ is a regular polynomial. It is easy to check that $e_H(g_l)=e_H(\bar g_l), e(g_l)=e(\bar g_l)$ and $u(\bar g_l) \leq u(g_l).$ Now observe that $\deg(\bar f_l) \leq u(\bar g_l) \leq u(g_l) \leq l$ and $S(h\bar f_l) =e(\bar g_l)=e(g_l)\leq l$. Therefore, $\Delta_h(l) \succeq S(\bar f_l) = e_H(\bar g_l) = e_H(g_l) \approx \Delta_H^{\z \wre \z}(l).$  

On the other hand,  let us choose any polynomials $f_l(x)$ such that $\deg f_l\le l,$
$S(hf_l)\le l,$ and $\Delta_h(l)=\Delta_{h,1}(l)= S(f_l).$  Then by Lemma \ref{t5},
$|f_l(x)w|_H \ge S(f) = \Delta_h(l)$ while $$|f_l(x)w|_{\z\wre\z}= |f_l(x)h(x)a|_{\z\wre\z} 
\le S(hf) + 2l \le 3l.$$  It follows that $\Delta_H^{\z \wre \z}(l) \succeq \Delta_h(l),$
and the lemma is proved.
\end{proof}



\section{Lower Bounds on Polynomial Distortion}

Given any polynomial $h=\sum_{j=0}^td_jx^j \in \z[x], d_0,d_t \ne 0$ with complex, real or integer coefficients, we are able to compute the equivalence class of its distortion function. 

\begin{lemma}\label{hnc}
The distortion $\Delta_h(l)$ of $h$ with respect to the ring of polynomials over $\z, \mathbb{R},$ or $\mathbb{C}$ is bounded from below by $l^{\kappa+1}$, up to equivalence, where $c$ is a complex root of $h$ of multiplicity $\kappa$ and modulus one.
\end{lemma}

\begin{proof}
Let $c$ be a complex root of $h$ of multiplicity $\kappa$ and modulus $1$. That is, $$h(x)=(x-c)^{\kappa}\tilde{h}(x)$$ over $\mathbb{C}$. Let $$v_l(x)=x^{l-1}+cx^{l-2}+\cdots+c^{l-2}x+c^{l-1}.$$ Then the product $$h(x)v_l^{\kappa+1}(x)=(x^l-c^l)^{\kappa}\tilde{h}(x)v_l(x)$$ satisfies $S(hv_l^{\kappa+1})=O(l)$, because $S(v_l)=O(l)$. On the other hand, because $|c|=1$, we have that $S(v_l^{\kappa+1}) \geq |v_l(c)^{\kappa+1}|=l^{\kappa+1}.$ This implies that if $c \in \mathbb{R}$; i.e. $c = \pm 1$, then $\Delta_h(l) \succeq l^{\kappa+1}$, where the distortion is considered over $\mathbb{C}$, $\mathbb{R}$ or over $\mathbb{Z}$. 

We will show that a similar computation holds over $\mathbb{R}$ and over $\z$ even in the case when $c \in \mathbb{C} - \mathbb{R}$. Let $\bar c$ be the complex conjugate of $c$. By hypothesis that $c \notin \mathbb{R}$ we know that $\bar c \ne c.$ Then $\bar c=c^{-1}$ is a root of $h(x)$ of multiplicity $\kappa$ as well, and $$h(x) = (x-c)^{\kappa}(x-\bar c)^{\kappa} H(x)$$ where $H(x)$ has real coefficients. Consider the product $v_l(x)\bar v_l(x),$ where $$\bar v_l(x) = x^{l-1}+ \bar c x^{l-2}+\cdots +\bar c ^{l-1}.$$ A simple calculation shows that each of the coefficients of this product is a sum of the form $$\sum_{i+j=k} c^i\bar{c}^j= \sum_{i+j=k} c^{i-j}=c^{k} + c^{k-2}+\cdots+ c^{-k}.$$ This is a geometric progression with common ratio $c^2 \ne 1$. Therefore, the modulus of every such coefficient is at most $\frac{2}{|1-c^2|}$ and so $S(v_l\bar v_l)$ is $O(l)$. This computation implies that the products $$h(x)v_l^{\kappa+1}(x)\bar v_l^{\kappa+1}(x) = (x^l - c^l)^{\kappa} (x^l - \bar c^l)^{\kappa}H(x)v_l(x)\bar v_l(x)$$ have the sum of the modules of their coefficients which are $O(l)$.

The polynomial $v_l^{\kappa+1}(x)\bar v_l^{\kappa+1}(x)$ has real coefficients. There is a polynomial $F_l(x)$ with integer coefficients such that each coefficient of $F_l(x)- v_l^{\kappa+1}(x)\bar v_l^{\kappa+1}(x)$ has modulus at most one. Thus $S(hF_l)$ is also $O(l)$.
   
We will show that the sums of modules of coefficients of $F_l(x)$ grow at least as $l^{\kappa+1}$ on a subsequence from Remark \ref{poa}. It suffices to obtain the same property for $ v_l^{\kappa+1}(x)\bar v_l^{\kappa+1}(x).$ Since $|c|=1$,  we have that the sum of the modules of the coefficients of $ v_l^{\kappa+1}(x)\bar v_l^{\kappa+1}(x)$ is at least $$|v_l^{\kappa+1}(c)\bar{v_l}^{\kappa+1}(c)| = l^{\kappa+1}|\bar{v_l}^{\kappa+1}(c)|.$$ We will show that there exists a subsequence $\{l_i\}$ so that on this sequence, $$|\bar v_{l_i}^{\kappa+1}(c)| \geq \frac{1}{2}.$$ We have that $$\bar v_l(c) = c^{l-1}+c^{l-2}\bar c+\cdots+\bar c^{l-1}=c^{l-1}+ c^{l-3} +\cdots+ c^{1-l}$$ because $\bar c=c^{-1}$. Therefore $|\bar v_l(c)| = |1+c^2+\cdots+c^{2l-2}|$ and similarly, $|\bar f_{l+1}(c)| = |1+c^2+\cdots+c^{2l}|.$ One of these two numbers must be at least one half because $|\bar v_l(c) - \bar v_{l+1}(c)| = |c^{2l}| =1$. Thus either  $l$  or $l+1$ can be included in the sequence $\{l_i\}$, and all required properties are shown. 
    \end{proof}
    
    \section{Upper Bounds on Distortion of Polynomials}


In order to obtain upper bounds on distortion of polynomials we require some facts from linear algebra.
Fix an integer $k \geq 1$ and let $n>0$ be arbitrary. 

\begin{lemma}\label{ob}
Let $Y_1, \dots, Y_{n}, C_2, \dots, C_{n}$ be $k \times 1$ column vectors. Suppose that the sum of the modules of all coordinates of $C_2, \dots, C_{n}$ is bounded by some constant $c$, and that the modulus of each coordinate of $Y_1$ and $Y_{n}$ is also bounded by $c$. Suppose further we have the relationship 
\begin{equation}\label{mnt}
Y_i=AY_{i-1}+C_i, i=2, \dots, n 
\end{equation}
where $A$ is a $k \times k$ matrix, in Jordan normal form, having only one Jordan block. Then the modulus of each coordinate of arbitrary $Y_i, 2 \leq i \leq n-1$ is bounded by $dcn^{k-1}$ where $d$ depends on $A$ only. In the case where the eigenvalue of $A$ does not have modulus one, the modulus of each coordinate of arbitrary $Y_i, 2 \leq i \leq n-1$ is bounded by $cd$, where $d$ depends on $A$ only. All matrix entries are assumed to be complex.
\end{lemma}

\begin{proof}
Let $\lambda$ be the eigenvalue of $A$, so that $A=\begin{pmatrix}
\lambda & 0 & 0  \hdots & 0  \\
1 & \lambda & 0 \hdots & 0\\
\vdots & \ddots & \ddots & \vdots \\
0 & 0  \hdots  & 1 & \lambda \\
\end{pmatrix}.$

We will consider cases.

\begin{itemize}
\item First suppose that $|\lambda| < 1$. 
\end{itemize}
From Formula \eqref{mnt} we derive:
\begin{equation}\label{zx}
Y_i=A(AY_{i-2}+C_{i-1})+C_i=(A)^2Y_{i-2}+AC_{i-1}+C_{i} = \cdots $$ 
$$= (A)^{i-1}Y_1+(A)^{i-2}C_2+ \cdots + AC_{i-1}+C_i.
\end{equation}
The following formula for $A^r$ is well-known because $A$ is assumed to be a Jordan block; it may also be checked easily using induction. We have that 
$$A^r=
\begin{pmatrix}
\lambda^r & 0 & 0  \hdots & 0  \\
r\lambda^{r-1} & \lambda^r& 0 \hdots & 0\\
\frac{r(r-1)}{2!}\lambda^{r-2} & r\lambda^{r-1} & \lambda^r \hdots & 0\\
\vdots & \ddots & \ddots & \vdots \\
\frac{r!}{(r-(k-1))!(k-1)!}\lambda^{r-(k-1)}   \hdots & \frac{r(r-1)}{2!}\lambda^{r-2} & r\lambda^{r-1} & \lambda^r \\
\end{pmatrix},
$$ with the understanding that if $r < k-1$, any terms of the form $\binom{r}{j}\lambda^{r-j}$ where $r<j$ are $0$.
Arbitrary nonzero element of the matrix $A^r$ is of the form $\binom{r}{j}\lambda^{r-j}$ for some $j \leq k-1$. 
Let $a_{s,t}(r)$ denote the $s,t$ entry of $A^r$. Then $a_{s,t}(r)$ is either zero or of the form $\binom{r}{j}\lambda^{r-j}$ for some $0\leq j\leq k-1$ depending on $s$ and $t$. Then $$\sum_{r=1}^i|a_{s,t}(r)| \leq \sum_{r=1}^{\infty}|a_{s,t}(r)|=\sum_{r=1}^{\infty}|\binom{r}{j}\lambda^{r-j}|$$ which is a constant depending on $A$ and not on $i$, because the series on the right is convergent when $|\lambda|<1$. Let $$c_1=\max_{1 \leq s,t \leq k}\{\sum_{r=1}^{\infty}|a_{s,t}(r)|\}.$$ Let $\bar{A}$ be the $k \times k$ matrix whose $s,t$ entry is $\sum_{r=1}^{\infty}|a_{s,t}(r)|$, and the column $\bar{C}$ be obtained by placing in the $j^{th}$ row the sum of the modules of the entries of the $j^{th}$ row of all $C_i$ and $Y_1$. Then every entry of $\bar{C}$ is bounded by $2c$. Observe that the modulus of every entry in the right side of \eqref{zx} is bounded by an entry of $\bar{A}\bar{C}$, which is in turn bounded by $2cc_1$, which does not include any power of $n$ at all.

\begin{itemize}
\item Let $|\lambda| > 1$. 
\end{itemize}
Because $\lambda^{-1}$ is an eigenvalue of $A^{-1}$, there exists a decomposition $A^{-1}=SJS^{-1}$ where $$J=\begin{pmatrix}
\frac{1}{\lambda} & 0 & 0  \hdots & 0  \\
1 & \frac{1}{\lambda} & 0 \hdots & 0\\
\vdots & \ddots & \ddots & \vdots \\
0 & 0  \hdots  & 1 & \frac{1}{\lambda} \\
\end{pmatrix}.$$ Letting $Y_i'=S^{-1}Y_i$ and $C_i'=S^{-1}C_i$ we see by Equation \eqref{mnt} that $$Y_{n-r}'=J^{r}Y_n'+J^{r}C_n'+ J^{r-1}C_{n-1}'+\cdots + JC_{n-r+1}',$$ for $r=1,\dots,n-2$. Observe that the sum of modules of coordinates of $Y_{n-r}'$ is less than or equal to $ksc$, where $s$ depends on $S$ (and hence on $A$) only. Similarly, the sum of all modules of all coordinates of $C_2',\dots,C_n'$ is bounded above by $ksc$. This case now follows just as the previous one to obtain constant upper bounds on the modules of the entries in $Y_2',\dots,Y_{n-1}'$. Finally, the modulus of any coordinate of $Y_{n-r}$ is bounded by $ks$ times the modulus of a coordinate of $Y_{n-r}'$.

\begin{itemize}
\item Let $|\lambda| = 1$. 
\end{itemize}

In this case, we have that $$|\binom{r}{j}\lambda^{r-j}| = \binom{r}{j} = \frac{r(r-1) \cdots (r-(j-1))}{j!}$$ $$\leq r(r-1) \cdots (r-(j-1)) \leq r^j \leq n^{k-1}.$$

It follows from Equation $(\ref{zx})$ that every entry of $Y_i$ is bounded above by $2cn^{k-1}$.
\end{proof}

\begin{lemma}\label{yx}
Let $Y_1, \dots, Y_{n}, C_2, \dots, C_{n}$ be $k \times 1$ column vectors. Suppose that the sum of the modules of all coordinates of $C_2, \dots, C_{n}$ is bounded by some constant $c$, and that the modulus of each coordinate of $Y_1$ and $Y_{n}$ is also bounded by $c$. Suppose further we have the relationship $$Y_i=AY_{i-1}+C_i, i=2, \dots, n $$where $A$ is a $k \times k$ matrix. Then the modulus of each coordinate of arbitrary $Y_i, 2 \leq i \leq n-1$ is bounded by $dcn^{\kappa-1}$ where $d$ depends on $A$ only, and $\kappa \leq k$ is the maximal size of any Jordan block of the Jordan form of $A$ having eigenvalue with modulus one.
\end{lemma}

\begin{proof}
There exists a Jordan decomposition, $A=SA'S^{-1}$.

Let $S^{-1}=(s_{i,j})_{1 \leq i,j \leq k}$ and let $s=\max |s_{i,j}|$. Then for $C_i'=S^{-1}C_i$ and $Y_i'=S^{-1}Y_i$ we have that 
\begin{equation}\label{s}
Y_i'=A'Y_{i-1}'+C_i'.
\end{equation}
By hypothesis, the sum of the modules of all coordinates of $C_2',\dots,C_n'$ is bounded by $ksc=c'$ and the coordinates of $Y_1'$ and $Y_{n}'$ are bounded by $c'$ as well. As we will explain, our problem can be reduced to the similar problem for $Y_i'$ in (\ref{s}). Suppose that the modules of coordinates of every $Y_i'$ are bounded by $dc'n^{\kappa-1}$ where $d$ depends on $A$ only. Then, letting $S=(\tilde{s_{i,j}})_{1 \leq i,j \leq k}$ and $\tilde{s}=\max|\tilde{s_{i,j}}|$ we have by definition of $Y_i'$ that arbitrary element of $Y_i$ has modulus bounded above by $k\tilde{s}dc'n^{\kappa-1}=d'cn^{\kappa-1}$ where $d'=k^2s\tilde{s}d$ only depends on $A'$, as required.

Lemma \ref{ob} says that if $A'$ has only one Jordan block, then the bound is constant if the eigenvalue does not have modulus one. Otherwise, we have in this case that $k=\kappa$ and the claim is true. If there is more than one Jordan block present in $A'$, the problem is decomposed into at most $k$ subproblems, each with only one Jordan block of size smaller than $k$. Again, we are done  by Lemma \ref{ob}.
\end{proof}

We will use Lemma \ref{yx} to prove the following fact, which requires establishing some notation prior to being introduced. Let $d_0, \dots, d_t \in \z$ where $d_0,d_k \ne 0$. Let the $(n+k) \times n$ matrix $M$ have $j^{th}$ column, for $j=1, \dots, n$, given by $[0, \dots, 0, d_0, d_1, \dots, d_k, 0, \dots, 0]^{T}$, where $d_0$ first appears as the $j^{th}$ entry in this $j^{th}$ column. Given the matrix $M$, we may also construct the matrix
\begin{equation}\label{8u}
A=
\begin{pmatrix}
0 & 1 & 0  \hdots & 0  \\
0 & 0 & 1 \hdots & 0\\
\vdots & \vdots & \ddots & \vdots \\
0 & 0 \hdots & 0 & 1 \\
a_1 & a_2  \hdots  & a_{k-1} & a_{k}, \\
\end{pmatrix}
\end{equation}
where $a_j=-\frac{d_{k-j+1}}{d_0}$, for $j=1,\dots,k$. Let $\kappa$ be the maximal size of a Jordan block of the Jordan form of $A$ having eigenvalue with modulus one.

\begin{lemma}\label{r9}
Suppose that $X=[x_1, x_2, \dots, x_n]^T$ is a solution to the system of equations $MX=B$, where $B=[b_1, b_2, \dots, b_{n+k}]^T$. Then it is possible to bound the modules of all coordinates $x_1, \dots, x_n$ of the vector $X$ such that $|x_i| \leq cbn^{\kappa-1}$ where $b=\sum_j\{|b_j|\}$ for $1 \leq j \leq n+k$ and the constant $c$ depends upon $d_0, \dots, d_k$ only. 
\end{lemma}

Prior to proving Lemma \ref{r9} we prove an easier special case.

\begin{lemma}\label{a}
It is possible to bound the coordinates $x_1, \dots, x_{k}$ of the vector $X$ from Lemma \ref{r9} from above by $b\tilde{\gamma}$ where $b=\sum_j\{|b_j|\}$ and $\tilde{\gamma}=\tilde{\gamma}(d_0, \dots, d_{k-1})$. 
\end{lemma}

\begin{proof}
By Cramer's Rule, we have the explicit formula that $$|x_i|=\bigg|\frac{\det(L_i)}{\det(L)}\bigg|$$ where $L$ is the $k \times k$ upper left submatrix of $M$ corresponding to the first $k$ equations, and $L_i$ is obtained by replacing column $i$ in $L$ with $[b_1, \dots, b_{k}]^T$. Because $\det(L)=d_0^{k}$, it suffices to show that the desired bounds exist for $\det(L_i);$ that is, we must show that there exists a constant $\tilde{\gamma}$ depending on $d_0, \dots, d_{k-1}$ only such that $|\det(L_i)| \leq b\tilde{\gamma}$ for $i=1, \dots, k$. By expanding along the $i^{th}$ column in $L_i$, we find that $$\det(L_i)=\pm b_1f_1(d_0, \dots, d_{k-1})\pm b_2f_2(d_0, \dots, d_{k-1}) \pm \cdots \pm b_{k}f_{k}(d_0, \dots, d_{k-1}),$$ where for each $i=1, \dots, k$, $f_i$ is a function of $d_0, \dots, d_{k-1}$ only obtained as the determinant of a submatrix containing none of $b_1, \dots, b_{k}$. The proof is complete by the triangle inequality.
\end{proof}

Note that the $|x_j|$ for $j=n-k+1, \dots, n$ are similarly bounded by $b\overline{\gamma}$ for the same $b$ and some $\overline{\gamma}=\overline{\gamma}(d_{0}, \dots, d_{k-1})$ as in Lemma \ref{a}. It is clear according to Lemma \ref{a} that we may assume that $|x_i| \leq b\gamma$ for the same $\gamma=\gamma(d_0, \dots, d_{k-1})$ for all $i=1, \dots, k, n-k+1, \dots, n$.

We proceed with the proof of Lemma \ref{r9}.
\begin{proof}
It suffices to obtain upper bounds for $|x_i|$ when $n-k \geq i \geq k+1$. 

For such indices, we have that $$d_{k}x_{i-k}+d_{k-1}x_{i+1-k}+ \cdots + d_0x_i=b_i.$$ In other words, $$x_i=\xi_i+a_1x_{i-k}+a_2x_{i+1-k}+\cdots+a_{k}x_{i-1},$$ where $\xi_i=\frac{b_i}{d_0}$ and $a_j=-\frac{d_{k-j+1}}{d_0}$. Let $X_i=[x_{i-k+1}, \dots, x_i]^{T}$ and let $\Xi_i=[0, \dots, 0, \xi_i]^T$. Then for the matrix $A$ of Equation \eqref{8u} we have the recursive relationship $$X_i=AX_{i-1}+\Xi_i$$ for $i=k, \dots, n$. Observe that the matrix $A$ depends on $d_0, \dots, d_k$ only, and that the sum of modules of the entries in all $\Xi_i$ are bounded by $\frac{b}{|d_0|}$. 

We see by Lemma \ref{a} that Lemma \ref{yx} applies to our situation. Therefore, the modules of coordinates of arbitrary $X_i$, $k+1 \leq i \leq n-k$ are bounded by $dc(n-k+1)^{\kappa-1}\leq dcn^{\kappa-1}$, where $d$ depends only on $d_0, \dots, d_k$, $c=\max\{\frac{b}{|d_0|},\gamma b\}$.
\end{proof}

\begin{lemma}\label{cz}
Let $h(x)=d_0+\dots+d_tx^t$, where $d_0,d_t \ne 0$. Then the distortion of $h$ is at most $l^{\kappa+1}$ where $\kappa$ is the maximal size of a Jordan block of the Jordan form of $$A=
\begin{pmatrix}
0 & 1 & 0  \hdots & 0  \\
0 & 0 & 1 \hdots & 0\\
\vdots & \vdots & \ddots & \vdots \\
0 & 0 \hdots & 0 & 1 \\
-\frac{d_t}{d_0} & -\frac{d_{t-1}}{d_0}   \hdots  & -\frac{d_2}{d_0}  & -\frac{d_1}{d_0}  \\
\end{pmatrix}
$$ of Equation \eqref{8u} with eigenvalue having modulus one. 
\end{lemma} 

\begin{proof} 
Consider any $f=\sum_{q=s}^{s+p}z_qx^q$ as in Definition \ref{pd}. Then consider $hf=\sum_{j=s}^{s+p+t}y_jx^j$. The coefficients $y_j$ are given by the matrix equation $MZ=Y,$ where $Z=[z_{s}, \dots, z_{s+p}]^T, Y=[y_{s}, \dots, y_{s+p+t}]^T$ and 
$$M=
\begin{pmatrix}
d_{0} & 0 & 0 & \hdots & 0  \\
d_{1} & d_{0} & 0 & \hdots & 0 \\
d_{2} & d_{1} & d_{0}& \hdots & 0\\
\vdots & \hdots & \ddots & \ddots & \vdots \\
d_{t} & d_{t-1} & \hdots & d_{1} \hdots & 0 \\
0 & d_{t} & \hdots & d_{2} \hdots & 0 \\
\vdots &  & \ddots &  & \vdots \\
0 & \hdots & 0 & d_{t} & d_{t-1} \\
0 & \hdots & 0 & 0 & d_{t} \\
\end{pmatrix}
$$
is an $(p+t+1) \times (p+1)$ matrix.

By Lemma \ref{r9} we have that for each $q=s, \dots, s+p$ that $|z_q| \leq cy(p+1)^{\kappa-1}$ where $c=c(d_0,\dots,d_t), y=\sum_j |y_{j}| \leq l.$ Therefore, $$\Delta_h(l) \leq S(f) = \sum_{q=s}^{s+p}|z_q| \leq c(l+1)^{\kappa+1}.$$

\end{proof}

    The following theorem shows that the upper and lower bounds are the same, and so we can compute exactly the distortion of a polynomial.
       
   \begin{theorem}\label{tgs}
   Let $h(x)=d_0+\dots+d_tx^t$ be a polynomial in $\z[x]$. Then the distortion of $h$ is equivalent to a polynomial. Further, the degree of this polynomial is exactly one plus the maximal multiplicity of a (complex) root of $h(x)$ having modulus one. 
   \end{theorem}
   
   \begin{proof}
   On the one hand, Lemma \ref{hnc} shows that the distortion is bounded from below by the 
   polynomial of degree one plus the maximal multiplicity $\kappa$ of a root of $h(x)$ having modulus one. 
On the other hand, the characteristic polynomial $\chi(x)$ of the matrix $A$ in Lemma \ref{cz} equals $x^t+\frac{d_1}{d_0}x^{t-1}+\cdots+\frac{d_{t-1}}{d_0}x+\frac{d_t}{d_0}=x^th(x^{-1})/d_0$.
And so the real polynomials $\chi(x)$ and $h(x)$ have the same roots with modulus $1$ (and with the
same multiplicities). Since the size of a Jordan block does not exceed
the multiplicity of the root of the characteristic polynomial, we have $\Delta_h(l) \preceq l^{\kappa +1}$ by Lemma \ref{cz}.  The theorem is proved.

\end{proof}

\begin{rem} Theorem \ref{tgs} will be used here for polynomials with integer coefficients,
but it is valid (with the same proof) for polynomials with complex or real coefficients. 
\end{rem}

Theorem \ref{tgs} and Lemma \ref{cz} imply the following.

\begin{cor} \label{exempl} The distortion of any exemplary subgroup $H$ of $\z\wre\z$  is equivalent to a polynomial. The degree of this polynomial is exactly one plus the maximal multiplicity of a (complex) root having modulus one of the polynomial $h(x)$ associated with $H$. 
\end{cor}

\section{Tame Subgroups}

For every $k\ge 1,$ the wreath product $\z\wre\z$ has subgroups $W \lambda \langle b^k\rangle$ isomorphic to $\z^k \wre \z,$ and so we are forced to study distortion in the groups $\z^k \wre \z$ 
even we are interested in $\z \wre \z$ only. Let $a_1,\dots,a_k; b$ be canonical generators of $\z^k \wre \z.$  If a subgroup $H$ of $G=\z^k \wre \z$ is generated by $b, w_1,\dots, w_k,$ where every
$w_i$ belongs to the normal closure $W_i$ of $a_i$ ($W_i$ = the submodule $\z[\langle b\rangle]a_i$
of $W$) then we say that $H$ is a {\it tame} subgroup of  $\z^k \wre \z.$

If $w_i\ne 1,$ then the subgroup $H_i$ is an exemplary subgroup of the wreath product
$G_i = W_i\lambda \langle b\rangle \cong \z \wre \z.$



\begin{lemma}\label{lcc}
For the tame subgroup $H$, we have that $$\Delta_H^G(l) \approx \sum_{i=1}^k \Delta_{H_i}^{G_i}(l).$$ 
\end{lemma}

\begin{proof}
Observe that $H_i \hookrightarrow H$ is an undistorted embedding, due to that fact that $H_i$ is a retract of $H$ (and similarly for $G_i \hookrightarrow G$). Therefore, by Lemma \ref{wkf} we have that $$\Delta_{H_i}^{G_i}(l) \preceq \Delta_{H_i}^G(l) \preceq \Delta_H^G(l),$$ for every $i$, and therefore  $ \Delta_H^G(l)\succeq \sum_{i=1}^k  \Delta_{H_i}^{G_i}(l).$ 

To prove the other inequality, we 
consider an element $u=vb^t\in H$ with $|u|_G\le l.$  Then there is a unique decomposition
$v=v_1+\dots +v_k,$ where
$v_i\in H_i,$ and for $u_i=v_ib^t,$ we have $u_i\in H_i$ since $H$ is tame. Then we have $|u_i|_{G_i}
\le |u|_G\le l$ since $G_i$ is a retract of $G.$    Therefore the required inequality
will follow from the inequality $|u|_H\le \sum_i |u_i|_{H_i}.$
This inequality is true indeed by Theorem \ref{t6} because $\trace_H(u)\le \sum_i \trace _{H_i}(u_i)$ since $\supp_H(u)\subset \cup_i \supp_{H_i} (u_i),$ and $||v||_H \le 
\sum_i ||v_i||_{H_i}$ since $H$ is a tame subgroup of $G.$

\end{proof}

\begin{cor} \label{tame} Every tame subgroup of $\z^k\wre\z$ has a polynomial
distortion.

\proof The statement follows from Corollary \ref{exempl} and Lemma \ref{lcc}. \endproof

\end{cor}

\section{Some Modules}\label{mt}
To get rid of the word `tame' in the formulation of Lemma \ref{tame}, we will need few remarks about modules. 
The following is well known (see also \cite{fs}).

\begin{lemma}\label{r1}
The ring $F[\langle b \rangle]$ is a principal ideal domain if $F$ is a field.
\end{lemma}

\begin{lemma}\label{ccc}
Suppose that $\overline{W}$ is a submodule of a free module $\overline{V}$ of rank $k$ over a (commutative) principal ideal ring $R.$ Then $\overline V$ is  a free module of  rank $l \leq k$, and modules $\overline{V}$ and $\overline{W}$ have bases $e_1', \dots, e_l'$ and $f_1', \dots, f_k'$ respectively such that for some $u_i' \in R,$ $$e_i'=u_i'f_i', i=1, \dots, l$$ 
\end{lemma}
 


At first we apply this statement to the following special case of Theorem \ref{x} Part $(2)$. 
\begin{lemma}\label{ta}
If $p$ is a prime, then any finitely generated subgroup $H$ of $G=\z_p^k \wre \z$ containing
the generator $b$ is undistorted.
\end{lemma}

\begin{proof} 

By Lemma \ref{wkf}, it suffices to show that $H$ has finite index in a retract $K$ of  $G.$

Since $p$ is a prime,  $\z_p$ is a field. This implies by Lemma \ref{r1}, that the ring $R=\z_p[\langle b \rangle]$ is a principal ideal ring.

Let $V=H\cap W.$ Then $V$ is a free $R$-module by Lemma \ref{ccc}, and we have that $V$ and $W$ have bases $e_1, \dots, e_m$ and $f_1, \dots, f_k$ respectively, for $m \leq k$ such that 
\begin{equation}\label{xc}
e_i=g_if_i, i=1, \dots, m
\end{equation}
for some polynomials $g_i \in R\backslash 0$. Thus we can choose the generators for $G$ and $H$ to be $\{b, f_1, \dots, f_l \}$ and $\{b, e_1, \dots, e_m\}$, respectively, and $H$ is a subgroup
of the retract $K$ of $G$, where $K$ is isomorphic to $\z_p^m \wre \z$ and is generated by
$\{b, f_1, \dots, f_m \}.$ Now $V$ is a submodule of the $\z_p[\langle b\rangle]$-module
$W'$ generated by $\{f_1, \dots, f_m \},$ and the factor-module $W'/V$ is a direct sum
of cyclic modules  $\langle f_i \rangle/\langle g_if_i\rangle.$ Hence $W'/V$
is finite since it is easy to see that each $\langle f_i \rangle/\langle g_if_i\rangle$ 
has finite order at most $p^{\deg g_i}.$ Since the subgroup $H$ contains $b,$ the index 
of $H$ in $K$ is also finite. 

\end{proof}

We return to our discussion of module theory. Let $H \leq \z^k \wre \z$ be generated by $b$, as well as any elements $w_1, \dots, w_k \in W$. Let $V$ be the normal closure of $w_1, \dots, w_k$ in $\z^r \wre \z;$ i.e., the $\z[\langle b\rangle]$-submodule of $W$ generated by $w_1,\dots, w_k.$ Let $\overline{V}= V \otimes_{\z} \q$ and $\overline{W} = W \otimes_{\z} \q$. Observe $\overline{W}$ and $\overline{V}$ are free modules over $\qb$ of respective ranks $k$ and $l \leq k$.

\begin{rem}\label{cb}
It follows from Lemma \ref{ccc} that there exist $0<m,n \in \z$ with $(me_i')=u_i(nf_i')$ where $e_i=me_i' \in V, f_i=nf_i' \in W, u_i \in \zb$. Moreover, the modules generated by $\{e_1, \dots, e_l\}$ and $\{f_1, \dots, f_k\}$ are free.
\end{rem}

\begin{rem}\label{ca}
There is a bijective correspondence between the set of finitely generated $\zb$ submodules $N$ of $\zb^k$ and the set of subgroups $K=N \lambda\langle b \rangle$ of $\z^k \wre \z$ such that the finite set of generators of $K$ is of the form $b, w_1, \dots, w_k$, $w_i \in W$. 
\end{rem}

\begin{rem}\label{bfz}
Let $V_1$ and $W_1$ be generated as submodules over $\zb$ by the elements from Remark \ref{cb}: $e_1, \dots, e_l$ and $f_1, \dots, f_k$ respectively. Let $H_1$ and $G_1$ be subgroups of $\z^r \wre \z$ generated by $\{b, V_1\}$ and $\{b, W_1\}$ respectively. It follows by Remark \ref{cb} that that $G_1 \cong \z^k \wre \z$ and $H_1 \cong \z^l \wre \z$. 

\end{rem}

\begin{rem}\label{bc}
Observe that under the correspondence of Remark \ref{ca} each generator $e_i$ of the group $H_1$ is in the normal closure of only one generator $f_i$ of $G_1$, i.e., $H_1$ is
a tame subgroup of $G_1.$
\end{rem}


\begin{lemma}\label{ba}
There exists $0<n', m' \in \mathbb{N}$ so that $n'W \subset W_1 \subset W,$ and $m'V \subset V_1 \subset V.$
\end{lemma}

\begin{proof}
By Remark \ref{ca} we have that $V$ is a finitely generated $\zb$ module with generators $w_1, \dots, w_k$. For each $w_i$, we have that the element $w_i \otimes 1 \in \overline{V}$. Therefore, by Lemma \ref{ccc}, there are $\lambda_{i,j} \in \qb$ so that $w_i=\sum_{j=1}^{l}\lambda_{i,j}e_j'.$ First observe that $mw_i=\sum_{j=1}^{l}\lambda_{i,j}e_j,$ because $e_i=me_i' \in V$.

Next, there exists $M_i \in \mathbb{N}$ so that $M_imw_i = \sum_{j=1}^{l} \mu_{i,j}e_j \in V_1$ where $\mu_{i,j} \in \zb$. Let $m'=M_1 \dots M_k m.$ Then for any $v \in V$, we have that $v=\sum_{i=1}^kv_iw_i$ where $v_i \in \zb,$ and therefore, $m'v \in V_1$ as required. A similar argument works for obtaining $n'$.
\end{proof}

\begin{lemma}\label{ab}
Let $\z^k \wre \z=G = W \lambda \langle b \rangle$ and let $K= \langle \langle w_1, \dots, w_s \rangle \rangle^{G} \leq G$ be the normal closure of elements $w_i \in W$. Suppose that there exists $n \in \mathbb{N}$ and a finitely generated subgroup $K' \leq K$ so that $nK \leq K'.$ Then $$\Delta_{\langle b, K' \rangle}^G(l) \approx \Delta_{\langle b, K \rangle}^G(l).$$
\end{lemma}

\begin{proof}
We will use the notation that $K_1=\gp \langle K, b \rangle, K_1' = \gp \langle K', b \rangle, K_1'' = \gp \langle nK, b \rangle.$
Observe that the mapping $\phi: G \rightarrow G: b \rightarrow b, w \rightarrow nw \textrm{ for } w \in W$ is an injective homomorphism which restricts to an isomorphism $K_1 \rightarrow K_1''$. An easy computation which uses Lemma \ref{t5} and the definition of $\phi$ shows that for any $g \in K_1$, we have that 
\begin{equation}\label{e1}
|g|_{G} \leq |\phi(g)|_{G} \leq n|g|_{G}
\end{equation} where the lengths are computed in $G$ with respect to the usual generating set $\{a_1, \dots, a_k,b\}$. 

Observe that under the map $\phi$ we have that 
\begin{equation}\label{e2}
\textrm{for } x \in K_1, |x|_{K_1}= |\phi(x)|_{K_1''},
\end{equation}
where the lengths in $K_1''$ are computed with respect to the images under $\phi$ of a fixed generating set of $K_1$.

By their definitions, we have the embeddings 
\begin{equation}\label{ac}
K_1'' \leq K_1' \leq K_1 \overset{\phi}{\hookrightarrow} K_1''.
\end{equation}

By Equation (\ref{ac}) there exists $k'>0$ depending only on the chosen generating sets of $K_1$ and $K_1'$ so that 
\begin{equation}\label{e3}
\textrm{for any } x \in K_1', |x|_{K_1} \leq k'|x|_{K_1'}.
\end{equation}
It also follows by  Equation (\ref{ac}) that there exists a constant $k>0$ depending only on the chosen generating sets of $K_1''$ and $K_1'$ so that 
\begin{equation}\label{xa}
\textrm{for any } x \in K_1'', |x|_{K_1'} \leq k|x|_{K_1''}.
\end{equation}

First we show that $\Delta_{K_1''}^{G}(l) \preceq \Delta_{K_1}^{G}(l).$

Let $g \in K_1''$ be such that $|g|_{G} \leq l$ and $|g|_{K_1''}=\Delta_{K_1''}^{G}(l)$. Then there exists $g' \in K_1$ such that $\phi(g')=g$. Therefore, it follows that $\Delta_{K_1''}^{G}(l) = |g|_{K_1''} =|\phi(g')|_{K_1''}=|g'|_{K_1} \leq \Delta_{K_1}^{G}(l).$ The first and second equalities follow by definition, the third by Equation (\ref{e2}), and the inequality  is true because by Equation (\ref{e1}) we have that $|g'|_G \leq |\phi(g)|_G=|g|_G \leq l.$

We claim that $\Delta_{K_1}^G(l) \preceq \Delta_{K_1'}^G(l).$

Let $g \in K_1$ be such that $|g|_{K_1}=\Delta_{K_1}^G(l)$. Then $|g|_{K_1} \leq |\phi(g)|_{K_1} \leq k'|\phi(g)|_{K_1'} \leq k'\Delta_{K_1'}^G(nl),$ which follows from Equations (\ref{e1}), (\ref{e3}) and by definition.

On the other hand, we will show that $\Delta_{K_1'}^G(l) \preceq \Delta_{K_1''}^G(l).$
Let $g \in K_1'$ be such that $|g|_{K_1'}=\Delta_{K_1'}^G(l)$. Then $|g|_{K_1'} \leq |\phi(g)|_{K_1'} \leq k|\phi(g)|_{K_1''} \leq k\Delta_{K_1''}^G(nl),$ which follows from Equations (\ref{e1}), (\ref{xa}) and by definition.

Therefore, we have that $\Delta_{K_1}^G(l) \preceq \Delta_{K_1'}^G(l) \preceq \Delta_{K_1''}^G(l) \preceq \Delta_{K_1}^G(l).$
\end{proof}

We say that $H$ is a {\it subgroup with $b$} in a wreath product $A \wre \langle b\rangle$ if
$H=\langle b,w_1,\dots,w_s\rangle$ where $w_1,\dots, w_s\in W$.

\begin{lemma}\label{addm}
Let $H$ be a subgroup with $b$ in $G=\z^k \wre \z$. Then the distortion of $H$ in $\z^k \wre \z$ is equivalent to the distortion of a tame subgroup $H_1$ of a wreath product $G_1=\z^l \wre\z,$ $l\le k.$ 
\end{lemma}

This follows from the results of Section \ref{mt}. Recall that the tame subgroup $H_1$ of the group $G_1$ was defined in Lemma \ref{bfz}, and these groups were associated to the given $H \leq G$. It follows from Lemmas \ref{ba} and \ref{ab} that  $$\Delta_{G_1}^{G}(l) \approx \Delta_{G}^{G}(l) \approx l \textrm{ and } \Delta_{H_1}^{G}(l) \approx \Delta_{H}^{G}(l),$$
and therefore $\Delta_{H_1}^{G_1}(l) \approx \Delta_{H}^{G}(l).$

\begin{cor}\label{withb} The distortion of every subgroup with $b$ in $\z^k \wre\z$ is polynomial.
\end{cor} 

\proof This follows from Corollary \ref{tame} and Lemma \ref{addm}. \endproof

 \section{Distortion in  $A \wre \z$}\label{vvv}

 In this section, we will reduce distortion in subgroups of $A \wre \z$ where $A$ is finitely generated abelian to that in subgroups of $\z^k \wre \z$ only. 




\begin{lemma}\label{grga}
Let $A$ be a finitely generated abelian group and consider $G=A \wre \z=A \wre \langle b \rangle.$ Assume that $k$ is the torsion-free rank of $a.$ If $H$ is a subgroup with $b$ in $G$ then the distortion of $H$ in $G$ is equivalent to that of a subgroup with $b$ in $\z^k \wre \z$.
\end{lemma}

\begin{proof}
There exists a series of subgroups $$A=A_0 > A_1 > \cdots > A_m \cong \z^k$$ for $k \geq 0$ where $A_{i-1}/A_i$ has prime order for $i=1, \dots, m$. 

We induct on $m$. If $m=0$, then $A \cong \z^k$ and the claim holds.

Now let $m>0$. Observe that $A_1$ is a finitely generated abelian group with a series $A_1 > \cdots > A_m \cong \z^k$ of length $m-1$. Therefore, by induction, any  subgroup with $b$ in $G_2=A_1 \wre \z$ has distortion equivalent to that of a subgroup with $b$ in $\z^k \wre \z$, for some $k$. 

By Lemma \ref{ta}, all subgroups with $b$ of $G_1=(A/A_1) \wre \z$ are undistorted. 
Denote the natural homomorphism by $\phi: G \rightarrow G_1.$ Let $$U=\displaystyle\bigoplus_{\langle b \rangle}A_1=\ker(\phi).$$ Observe that $U \cdot \langle b \rangle \cong G_2.$ The product is semidirect because $U$ is a normal subgroup which meets $\langle b \rangle$ trivially, and it is isomorphic to the wreath product by definition: the action of $b$ on the module $\displaystyle\bigoplus_{\langle b \rangle} A_1$ is the same. 


Let $R=\z[\langle b \rangle]$. Observe that $R$ is a Noetherian ring. This follows from basic algebra because $\z$ is a commutative Noetherian ring. Therefore, $W$ is a finitely generated module over the Noetherian ring $R$, hence is Noetherian itself. Thus, the $R$-submodule $H \cap U$ is finitely generated. 
Let $\{w_1', \dots, w_r'\}$ generate $H \cap U$ as a $R$-module. Let $\{b, w_1, \dots, w_s\}$ be a set of generators of $H$ modulo $U$; that is, the canonical images of these elements generate the subgroup $H_1=HU/U \cong H/H \cap U$ of $G_1$. Then the set $\{b, w_1, \dots, w_s, w_1', \dots, w_r'\}$ generates $H$. Furthermore, the collection $\{b, w_1', \dots, w_r'\}$ generates the subgroup $H_2= (H\cap U)\cdot \langle b \rangle$ of $G_2$.

Let $g \in H$ have $|g|_G \leq l$. Then the image $g_1=\phi(g)$ in $G_1$ belongs to $H_1$, because $g \in H$, and has length $|g_1|_{G_1} \leq l.$ 
It follows by Lemma \ref{ta} that $H_1$ is undistorted in $G_1.$ Therefore, there exists a linear function $f: \mathbb{N} \rightarrow \mathbb{N}$ (which does not depend on the choice of $g$) such that $|g_1|_{H_1} \le f(l).$ That is to say, there exists a product $P$ of at most $f(l)$ of the chosen generators $\{b, w_1, \dots, w_s\}$ of $H_1$ such that $P=g_1^{-1}$ in $H_1$. Taking preimages, we obtain that $gP \in U$.

Because $H$ is a subgroup of $G$, there exists a constant $c$ depending only on the choice of finite generating set of $H$ such that for any $x \in H$ we have that
\begin{equation}\label{bx}
|x|_G \leq c|x|_H.
\end{equation}

It follows by Equation (\ref{bx}) that
\begin{equation}\label{xxc}
|gP|_G \leq |g|_G + |P|_G \leq |g|_G+c|P|_H \leq l+cf(l).
\end{equation}

Observe that $gP \in H_2.$ This follows because $gP \in U$ by construction, and $g \in H$ by choice. Further, $P \in H$ because it is a product of some of the generators of $H$. Since $H_2=(H \cap U)\cdot \langle b \rangle$ we see that $gP \in H_2$. Using the fact that $G$ and $G_2$ are wreath products together with the length formula in Lemma \ref{t5}, we have that for any $x \in G_2,$
\begin{equation}\label{cx}
 |x|_{G_2} \leq |x|_G.
 \end{equation}

By induction, the finitely generated subgroup $H_2$ of $G_2$ has distortion function $F(l)$ equivalent to that of a  subgroup $\tilde{H}_2$ with $b$ in $\z^k \wre \z$ for some $k$. That is, $F(l)=\Delta_{H_2}^{G_2}(l) \approx \Delta_{\tilde{H_2}}^{\z^k \wre \z}(l)$. In particular, for any $x \in H_2$, 
\begin{equation}\label{ax}
|x|_{H_2} \leq F(|x|_{G_2}).
\end{equation}
Since $gP \in H_2$, we have that $$|gP|_{H_2} \leq F(|gP|_{G_2}) \leq F(|gP|_G) \leq F(l+cf(l)).$$ The first inequality follows from Equation (\ref{ax}), the second from Equation (\ref{cx}), and the last from Equation (\ref{xc}).

Because $H_2 \leq H$ there is a constant $k$ such that for any $x \in H_2, |x|_H \leq k |x|_{H_2}.$

Combining all previous estimates, we compute that $$|g|_H \le |gP|_H+|P|_H \le k|gP|_{H_2}+f(l)\le kF(l+cf(l))+f(l).$$ 
Thus, at this point we have shown that $\Delta_H^G(l) \preceq F(l) = \Delta_{H_2}^{G_2}(l),$ since $f$ is linear. On the other hand, $\Delta_{H_2}^G(l) = \Delta_H^G(l)$ by Lemma \ref{ab}. By Lemma \ref{wkf} we have that $\Delta_{H_2}^{G_2}(l) \preceq \Delta_{H_2}^{G}(l)$ and so $\Delta_H^G(l) \approx \Delta_{H_2}^{G_2}(l) \approx \Delta_{\tilde{H_2}}^{\z^k \wre \z}(l)$.

\end{proof}

\begin{cor}\label{AwrZ} For any finitely generated abelian group $A$, the distortion of every subgroup $H$ with $b$ in $A \wre \z$ is polynomial. $H$ is undistorted if $A$ is finite.
\end{cor}
\proof This follows from Lemma \ref{grga} and Corollary \ref{withb}.

\section{Completion of the Proof of Theorem \ref{x}}

\begin{lemma}\label{r4}
Let $G$ be a group having normal subgroup $W$ and cyclic $G/W = \langle bW \rangle$. Then any finitely generated subgroup $H$ of $G$ may be generated by elements of the form $w_1b^t, w_2, \dots, w_s$ where $w_i \in W$.
\end{lemma}

The proof is elementary and follows from the assumption that $G/W$ is cyclic.

\begin{rem}\label{z}
It follows that any finitely generated subgroup in $A \wre \z=W\lambda \langle b\rangle$ can be generated by elements \\$w_1b^t, w_2, \dots, w_s$ where $w_i \in W$.
\end{rem}

\begin{defn}
For a fixed finitely generated abelian group $A$ and any $t>0$, the group $L_t$ is the subgroup of $A \wre \z$ generated by the subgroup $W$ and by the element $b^t$. 
\end{defn}


\begin{lemma}\label{zz}
If $A$ is a fixed $r$ generated abelian group then $L_t \cong A^{t} \wre \z$,
where $A^t = A\bigoplus\dots\bigoplus A$ ($t$ times). 
\end{lemma}

\proof The statement follows from Remark \ref{yyy} with \\$A^t = A_1\bigoplus A_b\bigoplus\dots\bigoplus A_{b^{t-1}}$.\endproof

\begin{lemma}\label{zzz}
For any $w \in W$ there is an automorphism $L_t \rightarrow L_t$ identical on $W$ such that $wb^t \rightarrow b^t$, provided $t \ne 0$.
\end{lemma}

\proof This follows because the actions by conjugation of $b^t$ and $wb^t$ on $W$ coincide.
\endproof

\begin{lemma}\label{r5}
Let $H$ be a finitely generated subgroup of $A \wre \z$ not contained in $W$, where $A$ is finitely generated abelian. Then the distortion of $H$ in $A \wre \z$ is equivalent to the distortion of a subgroup $H'$ with $b$ in $A' \wre \z$ where 
$A' \cong A^t$ is also finitely generated abelian.
\end{lemma}

\begin{proof}
By Lemma \ref{r4} the generators of $H$ may be chosen to have the form $w_0b^t, w_1, \dots, w_s$ where $w_i \in W$. Therefore, for this value of $t$ we have that $H$ is a subgroup of $L_t$. Using the isomorphisms of Lemmas \ref{zz} and \ref{zzz} we have that $H$ is a subgroup of $A^{t} \wre \z=A' \wre \z$ generated by the image of $b^tw_0, w_1, \dots, w_s$ under the two isomorphisms: elements $b, x_1, \dots, x_s$. Finally, because $[A \wre \z: L_t ]<\infty$ we have by Lemma \ref{wkf} that the distortion of $H$ in $A \wre \z$ is equivalent to the distortion of its image in $A^{t} \wre \z$. 
\end{proof}




 {\bf Proof of Theorem \ref{x}}\label{final}

Theorem \ref{x} Parts $(1)$ and $(2)$ follow from Lemma \ref{abel} if the subgroup $H$ is abelian. Otherwise they follow from Corollary \ref{AwrZ} and Lemma \ref{r5}.


Now we complete the proof of Theorem \ref{x}, Part $(3)$. Let $A$ be a finitely generated abelian group of rank $k$. Consider the $2$-generated subgroup $H \leq \z \wre \z$ constructed as follows. Let $m \in \mathbb{N}$. Consider $h(x)=(1-x)^{m-1}$. Then the distortion of the polynomial $h$ is seen to be equivalent to $l^m$, by Lemma \ref{tgs}. By Lemma \ref{hpu}, this means that the $2$-generated subgroup $\langle b,(1-x)^{m-1}a \rangle = H_m \leq \z \wre \z$ has distortion $\Delta_{H_m}^{\z \wre \z}(l) \approx l^m$. The subgroup $\z \wre \z$ is a retract of $A \wre \z$ if $A$ is infinite. Therefore, the distortion of $H_m$ in $\z \wre \z$ and in $A \wre \z$ are equivalent by Lemma \ref{wkf}.


\begin{rem}
If we adopt the notation that the commutator $[a,b]=aba^{-1}b^{-1}$, then we see that in $\z \wre \z$, the element of $W$ corresponding to the polynomial $(1-x)^{m-1}a$ is $[ \cdots [a,b],b], \cdots, b]$ where the commutator is $(m-1)$-fold. This explains Corollary \ref{xx}.
\end{rem}

\medskip

\noindent {\bf Acknowledgement.}
The authors are grateful to Nikolay Romanovskiy for his many valuable comments.

 \addtolength{\textwidth}{.7in}
\addtolength{\evensidemargin}{-0.35in}
\addtolength{\oddsidemargin}{-0.35in}
\addtolength{\textheight}{.5in}
\addtolength{\topmargin}{-.25in}

\end{document}